\documentclass[a4paper,12pt]{amsart}

\usepackage[
backend=biber,
style=alphabetic,
sorting=nyt
]{biblatex}

\usepackage{amsfonts}
\usepackage{amsmath}
\usepackage{mathrsfs}
\usepackage{amssymb}
\usepackage{amsthm}
\usepackage{amscd}
\usepackage{latexsym}
\usepackage{amstext}
\usepackage{amsxtra}
\usepackage[utf8]{inputenc}
\usepackage[english]{babel}

\usepackage{geometry}
\usepackage{color}
\usepackage{paralist}

\usepackage[colorinlistoftodos]{todonotes}
\usepackage[all]{xy}
\usepackage{tikz}

\usepackage{enumerate}
\usepackage{multirow}
\usepackage{wasysym}
\usepackage{latexsym} 
\usepackage{comment}
\usepackage{colonequals}

\usepackage[
        draft=false,
        colorlinks, citecolor=darkgreen,
        pdfauthor={L.Fassina,G Pirola},
        pdftitle={},
        linktocpage        
]{hyperref}
\hypersetup{citecolor=blue,linktocpage}

\newtheorem{theorem}{Theorem}[section]
\newtheorem*{theorem*}{Theorem}
\newtheorem{lemma}[theorem]{Lemma}
\newtheorem{corollary}[theorem]{Corollary}
\newtheorem{proposition}[theorem]{Proposition}
\newtheorem{remark}[theorem]{Remark}
\newtheorem{definition}[theorem]{Definition}

\newcommand{\nc}{\newcommand} 
\nc{\cH}{{\mathcal H}}
\nc{\cA}{{\mathcal A}}
\nc{\cG}{{\mathcal G}}
\nc{\cC}{{\mathcal C}}
\nc{\cD}{{\mathcal D}}
\nc{\cO}{{\mathcal O}}
\nc{\cI}{{\mathcal I}}
\nc{\cB}{{\mathcal B}}
\nc{\cY}{{\mathcal Y}}
\nc{\cK}{{\mathcal K}} 
\nc{\cX}{{\mathcal X}}
\nc{\cS}{{\mathcal S}}
\nc{\cE}{{\mathcal E}}
\nc{\cF}{{\mathcal F}}
\nc{\cZ}{{\mathcal Z}}
\nc{\cQ}{{\mathcal Q}}
\nc{\cN}{{\mathcal N}}
\nc{\cP}{{\mathcal P}}
\nc{\cL}{{\mathcal L}}
\nc{\cM}{{\mathcal M}}
\nc{\cT}{{\mathcal T}}
\nc{\cW}{{\mathcal W}}
\nc{\cU}{{\mathcal U}}
\nc{\cJ}{{\mathcal J}}
\nc{\cV}{{\mathcal V}}
\nc{\bH}{{\mathbb H}}
\nc{\bA}{{\mathbb A}}
\nc{\bG}{{\mathbb G}}
\nc{\bC}{{\mathbb C}}
\nc{\bO}{{\mathbb O}}
\nc{\bI}{{\mathbb I}}

\nc{\bB}{{\mathbb B}}
\nc{\bY}{{\mathbb Y}}
\nc{\bK}{{\mathbb K}} 
\nc{\bX}{{\mathbb X}}
\nc{\bS}{{\mathbb S}}
\nc{\bE}{{\mathbb E}}
\nc{\bF}{{\mathbb F}}
\nc{\bZ}{{\mathbb Z}}
\nc{\bQ}{{\mathbb Q}}
\nc{\bN}{{\mathbb N}}
\nc{\bP}{{\mathbb P}}
\nc{\bL}{{\mathbb L}}
\nc{\bM}{{\mathbb M}}
\nc{\bT}{{\mathbb T}}
\nc{\bW}{{\mathbb W}}
\nc{\bU}{{\mathbb U}}
\nc{\bD}{{\mathbb D}}
\nc{\bJ}{{\mathbb J}}
\nc{\bV}{{\mathbb V}}
\nc{\bbZ}{{\mathbb Z}}
\nc{\bR}{{\mathbb R}}
\nc{\fr}{{\rightarrow}}
\nc{\co}{{\nabla}}
\newcommand{\debar}{{\bar\partial}}

\nc{\cu}{{\barline{\nabla}}}

\nc{\OO}{\mathcal{O}}
\nc{\PP}{\mathbb{P}}

\DeclareMathOperator{\Pic}{Pic}

\DeclareMathOperator{\codim}{codim}

\DeclareMathOperator{\Ker}{Ker}
\DeclareMathOperator{\Supp}{Supp}

\newcommand{\C}{\mathbb{C}}

\nc{\fA}{{\mathfrak{A}}}
\nc{\fB}{{\mathfrak{B}}}
\nc{\fC}{{\mathfrak{C}}}
\nc{\fD}{{\mathfrak{D}}}
\nc{\fE}{{\mathfrak{E}}}
\nc{\fF}{{\mathfrak{F}}}

\begin{document}


\title{A few remarks on sections of the Picard bundle of family of curves}
\date{\today}



\author{Lorenzo Fassina}
\address{Dipartimento di Matematica,
Universit\`a degli Studi di Pavia,
Via Ferrata 5,
27100 Pavia, Italy}
\email{lorenzo.fassina02@universitadipavia.it}

\author{Gian Pietro Pirola}
\address{Dipartimento di Matematica,
Universit\`a degli Studi di Pavia,
Via Ferrata 5,
27100 Pavia, Italy}
\email{gianpietro.pirola@unipv.it}


\maketitle

\begin{abstract}
    We study sections of the relative Picard bundle of a family of curves of genus $g \geq 2$ through the rank of the associated normal function. Using Griffiths’ formula for the infinitesimal invariant and higher Schiffer variations, we establish a numerical inequality relating the rank, the minimal support of a representing divisor, and the modular dimension of the family. When the modular map is dominant, we obtain a sharp classification: equality occurs only for multiples of odd theta characteristics or of the canonical section. As applications, we derive geometric consequences for plane curves, obtaining results on intersections with very general quartic curves, in the spirit of the work of Chen–Riedl–Yeong, and with quintic curves.
\end{abstract}


\section*{Introduction}

A \emph{normal function} is a  holomorphic section $\nu$ of the intermediate Jacobian fibration associated with a family of complex manifolds and satisfying the horizontality condition. Their study goes back to Griffiths, who introduced them as a tool for understanding how algebraic cycles vary under the Abel–Jacobi map. The behaviour of a normal function reflects subtle geometric information encoded in the cycles that determine the section.
In this paper we specialise this framework to the simpler, yet in our view still quite interesting, case of sections of the Picard bundle of a complex curve of genus $g \geq 2$. In particular, we investigate the rank, defined as the real dimension of its image under a local trivialisation of the Jacobian fibration. Very relevant results in this direction can be found in \cite{Betti_map}. Moreover in a recent work, Richard Hain (see \cite{Hain}) computed the rank of the normal function associated with the Ceresa cycle. 

A fundamental tool is the \emph{infinitesimal invariant} introduced by
Griffiths \cite{infinitesimal} and refined by Green \cite{Green} and Voisin \cite{Voisin-une-remarque}.  
It measures the obstruction to a normal function being locally constant. The subtle connections between the differential
associated to a normal function have been pointed out in the Hain paper.

Let $\pi: \mathcal{C} \to Y$ be a smooth family of complete complex curves of genus $g \geq 2$.
Given a section of the Picard bundle
\[
    \psi : Y \longrightarrow \mathrm{Pic}^n(\pi),
\]
we consider the associated normal function
\[
 \nu := (2g-2)\psi - n \kappa ,
\]
where $\kappa$ is the canonical section of $\mathrm{Pic}^{2g-2}(\pi)$.

In our case the rank measures
how far the normal function is from the canonical section modulo torsion. As in Hain’s work, our approach relies on the infinitesimal invariant. The principal tool is a refinement of a formula for divisors on curves established by Griffiths in his seminal article \cite{infinitesimal}. In order to apply Griffiths' formula we need to satisfy some striking conditions. First, the normal function must be represented by a suitable Weil divisor; additionally, a certain holomorphic form is required to vanish along the support of this divisor. It emerges that such a representation is feasible only when the support of the divisor is sufficiently small. Finally
to get substantial results we can apply the Griffiths' formula only on special tangent
directions of the moduli space that preserve the form of type $(1,0)$. Among the admissible directions, the Schiffer variations play a distinguished role, since they admit an explicit description in terms of the evaluation of holomorphic forms on the support of the divisor defining the normal function. These directions
have been studied quite systematically in \cite{Col-Fre-Pir}.

We denote by $d_S(\psi)$ the cardinality of the support of a divisor representing the section $\psi$ on a general fibre.  
Our main result shows that a section with maximal possible support and vanishing rank must be a multiple either of an odd theta characteristic or of the canonical section. In particular we can prove the following.



\begin{theorem}[= Theorem \ref{theo-spin-section}]
    Let $\pi : \mathcal{C} \to Y$ be a family of complex curves of genus $g \ge 2$. Assume that the image of the modular map $m: Y \to \mathcal{M}_g$ contains an analytic open subset.
    Let $\psi$ be a section of the Picard bundle $\mathrm{Pic}^n(\pi)$, and consider the associated normal function
    \[
        \nu = (2g-2)\psi - n\kappa .
    \]
    If $\nu$ is locally constant, then $d_S(\psi) = g-1$ if and only if, for a general point $y \in Y$ and every divisor $D$ such that $\psi(y) = \mathcal{O}_{C_y}(D)$ with $d_S(D) = d_S(\psi)$, there exists a holomorphic form $\omega \in H^0(C_y,\omega_{C_y})$ with divisor $K := \sum_{i=1}^{g-1} b_i p_i$ and support $\mathrm{Supp}(K) = \{p_1,\dots, p_{g-1}\}$, such that one of the following holds:
    \begin{enumerate}
        \item $D$ is a multiple of the divisor $\sum_{i=1}^{g-1} p_i$, i.e.\ $\psi$ is a multiple of an odd spin section;
        \item $D$ is a multiple of $K$, i.e.\ $\psi$ is a multiple of the canonical section.
    \end{enumerate}
    In either case, $\mathrm{Supp}(D) = \mathrm{Supp}(K)$.
\end{theorem}

These general results admit concrete and geometric applications. In forthcoming joint work with I. Biswas \cite{Bis_pir_fanzo}, we apply them to the study of the intersection between the symmetric theta divisor and torsion points on a very general Jacobian, and, via \cite{A-A-G-Vikram}, to the Dirac operator on a general even spin curve. For families of curves with small moduli dimension the resulting statements are necessarily weaker; nevertheless, in the final section we show that our computations still yield meaningful consequences for families of plane curves. In particular, we obtain the following result.

\begin{theorem}\label{prop_d_4_5 introduz}
    Let $C \subset \mathbb{P}^2$ be a very general plane curve of degree $d = 4$ or $5$, and let $D \subset \mathbb{P}^2$ be a plane curve. 
    Assume that $C$ and $D$ intersect in exactly $d-2$ distinct points.
    Then, these points lie either on a bitangent line or on a flex line of $C$.
    Moreover, let $\psi \in \mathrm{Pic}^n(\pi)$ be the section of the Picard bundle of the family of plane curves such that associated normal function $\nu$ is locally constant. Then
    \[
    d_S(\psi) \geq d-2.
    \]
\end{theorem}

Moreover, we conjecture that the same statement holds for every degree \(d > 5\). This formulation recovers and provides a conceptual framework for geometric phenomena previously investigated by Xu \cite{Xu}, who established a lower bound on the intersection of two plane curves and then further proved in degree $4$ in \cite{C-R-Y}.
For quartic curves, the result follows as an application of our earlier analysis of the rank of the normal function. The case of quintic curves is more constructive and relies on the explicit construction of suitable Schiffer variations. We hope to be able to extend this result for $d>5$. For this, it is necessary to
improve our knowledge of this infinitesimal variation of the smooth plane curve.
More generally, we believe that this method could give some new information on the
algebraic hyperbolicity of surfaces.

The paper is organised as follows.
Section~1 reviews the construction of the relative Picard bundle.
Section~2 discusses normal functions and their rank.
Section~3 recalls the infinitesimal invariant of Griffiths and its dual interpretation.
Section~4 introduces higher Schiffer variations.
Section~5 contains Griffiths' formula for decomposable tensors and first applications.
Section~6 contains the proof of the main inequality and its consequences.
Section~7 applies the theory to quartic and quintic plane curves.

\section*{Acknowledgements}
The authors are grateful to A.~Beauville and C.~Voisin for useful conversations and insightful remarks that contributed to the development of this work.



\section{The Relative Picard Bundle}
 Let $\pi: \cC\to Y$ be a smooth family of complete complex curves of genus $g \geq 2$, where $Y$ is a connected complex variety. We define $C_y := \pi^{-1}(y)$ as the fiber over $y \in Y$ and set $C := C_0$ for a base point $0 \in Y$. Let $\cM_g$ be the moduli space of algebraic curves of genus $g$ and consider the modular map $m: Y\to \cM_g$ defined by
\[
    m(y)=[C_y],
\]
that sends each point $y \in Y$ to the isomorphism class of the fiber $C_y$.

The \emph{Picard bundle} relative to the family $\pi: \cC\to Y$ is defined as
\[
  \mathrm{Pic}(\pi) := R^1 \pi_* \mathcal{O}_{\mathcal{C}}^{*}.
\]
This is a fiber bundle over $Y$, whose fiber at a point $y \in Y$ is the Picard group of the corresponding fiber:
\[
    \mathrm{Pic}(\pi)_y = \mathrm{Pic}(C_y).
\]
This decomposes into connected components indexed by the degree:
\[
  \mathrm{Pic}(\pi) = \coprod_{d \in \mathbb{Z}} \mathrm{Pic}^d(\pi).
\]

We will consider the holomorphic sections of $\pi_d: \mathrm{Pic}^d(\pi) \to Y$, where $\pi_d$ is the relative degree-$d$ Picard bundle, whose fiber over $y \in Y$ is the degree-$d$ component $\mathrm{Pic}^d(C_y)$ (for more details
about the Picard bundle refer to \cite[Chap.~4, Sec.~2]{ACG1}).
An important example is the canonical section. This is defined by considering, for each fiber curve $C_y$, its canonical line bundle $\omega_{C_y}$ of degree $2g-2$. This defines a global section
\[
  \kappa: Y \to \mathrm{Pic}^{2g - 2}(\pi), \qquad \kappa(y) := \omega_{C_y},
\]
of the degree-$2g-2$ Picard bundle. 

A further relevant instance is provided by sections
\(\psi : Y \to \mathrm{Pic}^{g-1}(\pi)\) of the $(g-1)$-Picard bundle satisfying
\[
  2\cdot \psi = \kappa,
\]
referred to as \emph{spin sections}. These are said to be \emph{even} or
\emph{odd} according to the parity of \(h^0(C_y,\psi(y))\).

A more general construction is obtained as follows. Let $\mathcal{D} \subset \mathcal{C}$ be a relative divisor such that for every $y \in Y$ the divisor $D_y := (\mathcal{D} \cdot C_y)$ is a divisor of $C_y$. More precisely, let 
\[
    \iota_y : C_y \rightarrow \mathcal{C}
\]
be the inclusion, we define $\mathcal{O}_{C_y}(D_y) := \iota_y^\ast(\mathcal{O}_{\mathcal{C}}(\mathcal{D}))$. Then let
\begin{equation}\label{ex relative divisor}
    \psi_{\mathcal{D}}(y)= \mathcal{O}_{C_y}(D_y) 
\end{equation}
be the associated section of the Picard bundle. 


 \begin{definition}
    For a divisor $D = \sum_{i=1}^k a_ip_i$ on the curve $C$, with $p_i \neq p_j$ and $a_i \neq 0$, let the support of $D$ be the set of points
    \[
        \mathrm{Supp}(D) := \{p_1,\dots, p_k\}.
    \]
     Define the degree of the support of a divisor $D$ of $C$ as the cardinality of its support
    \[
        d_S(D):= \#\mathrm{Supp}(D) = k.
    \]
    For a line bundle $L$ on $C$ we define 
    \[
        d_S(L) := \min_{L \cong \mathcal{O}_C(D)} d_S(D).
    \] 
\end{definition}

\begin{definition}
    Let $\psi: Y \rightarrow \mathrm{Pic}^d(\pi)$ be a section of the Picard bundle. We define the degree of the support of $\psi$ as
    \[
        d_S(\psi) = d_S(\psi(y))
    \]
    where $y \in Y$ is a general point.
\end{definition}

\begin{remark}
    With this definition for a general section $\psi$ of $\mathrm{Pic}^d(\pi)$ we have $d_S(\psi) \leq g$ by the Jacobi inversion theorem. For the canonical section we have $d_S(\kappa) \leq g-1$; this can be seen taking an odd theta characteristic or taking a canonical section with maximal vanishing at a Weierstrass point.
\end{remark}

\section{Normal functions}
\subsection{Basic theory of normal functions}

When $d = 0$, the degree-zero component of the Picard bundle $\mathrm{Pic}^0(\pi)$ is isomorphic to the Jacobian fibration $J(\pi)$ defined by
$$
 \frac{\mathcal{H}}{\mathcal{F} + R^1\pi_* \mathbb{Z}} \longrightarrow Y,
$$
where $\mathcal{H} := R^1\pi_* \mathbb{C}$ is the holomorphic vector bundle whose fiber at $y \in Y$ is 
$$
\mathcal{H}_y = H^1(C_y, \mathbb{C}),
$$ 
and $\mathcal{F}$ is the is the holomorphic subbundle whose fiber is 
$$
\mathcal{F}_y = H^0(C_y, \omega_{C_y}).
$$

Sections of $J(\pi)$ are classically known as \emph{normal functions} (see \cite[Chap.~2, Sec.~7]{Voisin} and \cite[Chap.~9]{Carlson} for a more general definition).

Let $\mathcal{D} \subset \mathcal{C}$ be a relative divisor such that for every $y \in Y$ the divisor $D_y$ of $C_y$ has degree zero. Then, we can define a normal function
\[
    \nu : Y \rightarrow \frac{\mathcal{H}}{\mathcal{F} + R^1\pi_* \mathbb{Z}}
\]
setting $\nu(y) = \mathcal{O}_{C_y}(D_y) \in \operatorname{Pic}^0(C_y)$, the Abel-Jacobi map.

A basic example of a normal function, that will be used later, is the following. Let
\[
\psi: Y \to \mathrm{Pic}^n(\pi)
\]
be a global section of the degree-$n$ Picard bundle. Then we define the \emph{associated normal function}
\[
\nu := (2g - 2)\psi - n\kappa,
\]
where $\kappa$ is the canonical section.

\subsection{Rank of a normal function}

Let $J(\pi)$ be the jacobian fibration. There is an isomorphism of smooth group bundles 
\[
    J(\pi) \stackrel{f}\rightarrow \frac{R^1\pi_*\mathbb{R}}{R^1\pi_*\mathbb{Z}}.
\]
For any simply connected open subset $U \subset Y$, the restriction of $J(\pi)$ to $U$ is isomorphic, as a smooth group bundle, to $T^{2g} \times U$, where $T^{2g} = \mathbb{R}^{2g}/\mathbb{Z}^{2g}$ denotes a real torus of dimension $2g$, and this local trivialization preserves the group structure on each fibre. Consider the map $\phi:= p_1 \circ f_{|_U}$
$$
J(\pi)_{|_U} \stackrel{f_{|_U}}{\longrightarrow} T^{2g} \times U \stackrel{p_1}{\longrightarrow} T^{2g},
$$
given by the composition of the restriction of $f$ to $U$ with the projection on the first factor $p_1$.

\begin{definition}\label{def rank}
    Let $\nu: Y \rightarrow J(\pi)$ be a normal function and $\nu_{|U}$ its restriction to $U$. Define the rank of $\nu$ as 
    \[
        \mathrm{rk} (\nu) = \frac{\dim((\phi\circ \nu )(U))}{2}
    \]
    where $\dim$ is the real topological dimension.
\end{definition}

\begin{remark}
    Following Hain \cite[Sec.~3]{Hain} and \cite{Betti_map}, the rank of $\nu$ is always an integer. Indeed, the fibers of $\phi\circ \nu$ are complex subvarieties of $J(\pi)_{|U}$ and then the dimension of the image is even.
\end{remark}

\begin{definition}\label{def-local-const}
    A normal function $\nu$ defined on the jacobian fibration $J(\pi)$ is said to be locally constant if 
    \[
    \mathrm{rk} (\nu) = 0.
    \]
\end{definition}

\begin{remark}
    An important example of a locally constant normal function is given by a torsion section.
\end{remark}

\section{The Griffiths infinitesimal invariant}
In this section, we review and elaborate on work of Griffiths\cite{infinitesimal}, Green\cite{Green} and Voisin\cite{Voisin-une-remarque} on invariants of normal functions associated to a family of curves. 

\subsection{Definition of the infinitesimal invariant}

The Griffiths infinitesimal invariant is an obstruction to a normal function being locally constant. Let us make it more precise. Recall that $\mathcal{H}$ is a flat vector bundle and hence it is equipped with a flat connection $\nabla$
\[
\nabla : \mathcal{H} \rightarrow \mathcal{H} \otimes \Omega^1_Y
\]
called the Gauss-Manin connection.

Set $\mathcal{H}^{0,1} := \frac{\mathcal{H}}{\mathcal{F}} \cong \mathcal{F}^*$. Consider the bundle map, representing the IVHS map,
\[
    \overline{\nabla} : \mathcal{F} \to \mathcal{H}^{0,1} \otimes \Omega_Y^1,
\]
defined, over a point $y \in Y$, by
\[
    \overline\nabla_\xi(\omega) := \kappa_y(\xi) \cdot \omega \in H^1(\mathcal{O}_{C_y}) \cong H^{0,1}(C)
\]
where $\kappa_y: T_{Y,y} \rightarrow H^1(T_{C_y})$ is the Kodaira-Spencer map and $\cdot$ is the cup product.
Following Voisin \cite[Chap.~2, Sec.~7]{Voisin} and Green~\cite{Green}, given a normal function $\nu$, we define the infinitesimal invariant. Let $\tilde{\nu}$ a local lifting of $\nu$ in $\mathcal{H}$, and consider the projection $\overline{\nabla \tilde{\nu}}$ of $\nabla \tilde{\nu}$ in the quotient
\[
    \frac{\mathcal{H}}{\mathcal{F}} \otimes \Omega^1_Y \cong \mathcal{H}^{0,1} \otimes \Omega^1_Y.
\]
Then consider the class
\[
    [\overline{\nabla \tilde{\nu}}] \in \frac{\mathcal{H}^{0,1}\otimes \Omega^1_Y}{\overline\nabla \mathcal{F}}.
\]
This is well defined and does not depend on the choice of lifting (see Lemma~7.7 in \cite{Voisin}). We denote by $\delta\nu$ this class and call it the infinitesimal invariant. It is useful for the computation to define also the dual version of $\delta\nu$. Let us consider the transpose of the Gauss–Manin connection
\[
    \overline{\nabla}^t : \mathcal{F} \otimes T_Y\to \mathcal{H}^{0,1}
\] 
Choose a point $y \in Y$ , then $\overline{\nabla}^t : H^0(C_y, \omega_{C_y}) \otimes T_{Y,y} \to H^1(C_y, \mathcal{O}_{C_y})$ is given by
\begin{equation}\label{Gauss_Manin}
    \overline{\nabla}^t \left( \sum_i \omega_i \otimes \xi_i \right) = \sum_i \overline{\nabla}_{\xi_i}(\omega_i) = \sum_i \kappa_y(\xi_i) \cdot \omega_i.
\end{equation}
By duality we have
\[
    \frac{\mathcal{H}^{0,1} \otimes \Omega_Y}{\overline\nabla \mathcal{F}} \cong \left(\ker(\overline{\nabla}^t)\right)^*,
\]
then we can consider $\delta\nu$ as a section of this last vector bundle. Over a point $y \in Y$, we have:
\begin{equation}\label{integrale inf inv}
    \delta\nu(y)\left( \sum_i \omega_i \otimes \xi_i \right) = \sum_i \int_C \nabla_{\xi_i} \tilde{\nu}\wedge \omega_i,
\end{equation}
where $\sum_i \kappa_y(\xi_i) \cdot \omega_i = 0$.

\subsection{Normal function of rank 0}
Building on Definition \ref{def-local-const}, we can now provide an equivalent formulation for a normal function to be of rank zero, which we wish to further refine.

\begin{proposition}
    A normal function $\nu$ is locally constant if and only if there exists a local lifting $\tilde{\nu} \in \mathcal{H}$ of $\nu$ such that $\nabla \tilde{\nu} = 0$.
\end{proposition}

For future reference, and following Hain \cite[Sec.~3]{Hain}, we provide a criterion for determining the rank of a normal function.

\begin{definition}
Fixed a point $y \in Y$, a normal function $\nu$ is said to be \emph{locally constant} along a vector $\xi \in T_{Y,y}$ if there exists a local lifting $\tilde{\nu} \in \mathcal{H}$ of $\nu$ near $y \in Y$ such that
\[
    \nabla_{\xi} \tilde{\nu} = 0.
\]
\end{definition}

\begin{proposition}\label{loc-cost-lungo-xi}
    Fixed a point $y \in Y$, a normal function $\nu$ has rank $\mathrm{rk}(\nu) \geq r$ if and only if there exist $r$ independent vectors $\xi_1,\dots,\xi_r \in T_{Y,y}$ such that for any non zero vector $\xi \in \langle \xi_1, \dots, \xi_r \rangle$ and for any local lifting $\tilde{\nu} \in \mathcal{H}$ near $y \in Y$,
    \[
        \nabla_\xi \tilde{\nu} \neq 0.
    \]
\end{proposition}

\begin{remark}\label{remark loc const}
    From the definitions follows that if the rank of a normal function vanishes then the infinitesimal invariant $\delta\nu$ vanishes as well. In particular, fixed a point $y \in Y$, if there exists a vector $\xi \in T_{Y,y}$ and a holomorphic form $\omega \in H^0(C_y, \omega_{C_y})$ such that $\kappa_y(\xi) \cdot \omega = 0$ and $\nu$ is locally constant along $\xi$, then
    \[
    \delta \nu(y) (\xi \otimes \omega) = 0.
    \]
\end{remark}

We now recall a characterization of locally constant normal functions, including a proof for completeness.

\begin{proposition}
Let $\pi: \mathcal{C} \to Y$ be an algebraic family of smooth curves of genus $g > 1$, and let $y \in Y$ be a point with fiber $C = C_y$. Assume that the modular map $m: Y \to \mathcal{M}_g$ is dominant, more generally, that the image of the monodromy representation
$$
\mu: \pi_1(Y, y) \to \mathrm{Aut}(H^1(C, \mathbb{R}))
$$
has finite index in $\mathrm{Sp}(2g, \mathbb{Z})$. Then, any locally constant normal function $\nu: Y \to J(\pi)$ is a torsion section.
\end{proposition}

\begin{proof}
Let $Y' \to Y$ be the universal cover. The map $\nu': Y' \to T^{2g} = H^1(C, \mathbb{R}) / H^1(C, \mathbb{Z})$ induced by the pullback is constant, so let $x = \nu'(Y')$. Any two liftings $x', x'' \in H^1(C, \mathbb{R})$ of $x$ differ by an element of the lattice $H^1(C, \mathbb{Z})$. For a loop $[\gamma] \in \pi_1(Y, y)$, let $\tilde{\gamma}$ denote a lifting of $\gamma$ to $Y'$, and let $\tilde{\nu}'$ be a lifting of $\nu'$ to $H^1(C, \mathbb{R})$. Then
$$
\tilde{\nu}'(\tilde{\gamma}(1)) - \tilde{\nu}'(\tilde{\gamma}(0)) \in H^1(C, \mathbb{Z}).
$$

By the Picard–Lefschetz formula, the monodromy action on $H^1(C, \mathbb{R})$ is given by
$$
L_\beta(\alpha) = \alpha + \langle \alpha, \beta \rangle \beta,
$$
for some $\beta \in H^1(C, \mathbb{Z})$ and any $\alpha \in H^1(C, \mathbb{R})$, where $\langle \cdot, \cdot \rangle$ denotes the intersection pairing.

By hypothesis, there exists an integer $n > 0$ such that $L_\beta^n$ corresponds to the monodromy of $\gamma^n \in \pi_1(Y, y)$, and we compute:
$$
L_\beta^n(x') - x' = n \langle x', \beta \rangle \beta \in H^1(C, \mathbb{Z}).
$$
This implies that $\langle n x', \beta \rangle \in \mathbb{Z}$ for all $\beta \in H^1(C, \mathbb{Z})$, so $n x' \in H^1(C, \mathbb{Z})$. Thus, $x$ is a torsion point in the torus $T^{2g}$, and $\nu$ is a torsion section.
\end{proof}

We recall a fundamental result, known as the strong Franchetta conjecture, which states that any rational section of the Picard variety of the universal curve over the moduli space of curves (or, more generally, of a universal family of varieties) is a multiple of the canonical bundle (see \cite{Harer}, \cite{maestrano-franc-conj}, \cite{Kouvidakis} and \cite{Arbarello-Cornalba_picard_groups} for more details).

\section{Cohomology and higher-Schiffer variations}

In this section, we describe the decomposable elements in $\ker(\overline{\nabla}^t)$, which will be used to compute the infinitesimal invariant.

Let $\omega_C$ be the canonical line bundle on a projective curve $C$, and $\omega \in H^0(C,\omega_C) $ be a non zero holomorphic form and let $Z$ be the zero scheme defined by $\omega$. Set
\[
    H^1(C, T_C)_\omega :=\{\eta \in H^1(C,T_C) : \eta\cdot \omega=0\}
\]
the space of infinitesimal deformations that preserve the $(1,0)$-form $\omega$ to first order.
The cohomology of the sequence
\begin{equation}\label{tangent seq}
    0 \to T_C \xrightarrow{\cdot\,\omega} \mathcal{O}_C \to \mathcal{O}_Z \to 0.
\end{equation}
gives
\begin{equation}\label{cohom seq}
    0\to H^0(C,\cO_C)\to H^0(\cO_{Z})\stackrel{\partial_Z}{\to} H^1(C, T_C)\stackrel{\cdot\omega}\to H^1(C,\cO_C)\to 0.
\end{equation}

In particular we get
\[
    \partial_Z(H^0(\cO_Z))= H^1(C, T_C)_\omega
\]
and we have 
\[
    \dim H^1(C, T_C)_\omega= 2g - 3.
\]
We may represent, as in \cite{Col-Fre-Pir}, this in Dolbeault cohomology. If
$\theta\in H^0(\cO_Z)$ and $\rho \in\cC^\infty(C)$ the bump function which is $1$ around $Z$ and $0$ outside a 
neighborhood of $Z$, then $\rho \cdot \theta$ is a $\mathcal{C}^\infty$ function. Therefore letting $\frac{1}{\omega}$ be the meromorphic section of $T_C$ we have in Dolbeault cohomology
\[
    \partial_Z \theta= \left[\frac{\debar (\rho \theta)}{\omega}\right]_{Dol}.
\]

\begin{definition} 
If $\theta\in H^0(\cO_Z)$ is such that $\theta(q)= 1$ at a point $q\in Z$ and $0$ on the other points,
we call
\begin{equation} \label{schiffer}
\zeta_q= \left[\frac{\debar \rho}{\omega}\right]_{Dol}
\end{equation}
a higher Schiffer cohomology class associated to $q$.
\end{definition}

We give a lemma that will be used later.

\begin{lemma}\label{schiffer indip}
    Let $C$ be a curve, $p \in C$ a point and $\omega \in H^0(C,\omega_C)$ a holomorphic form. Suppose that
    $\mathrm{ord}_p(\omega) = n < 2g-2$. Then the cohomology classes
    \[
    \zeta_j = \left[ \frac{\bar \partial \rho}{z^{j}}\right] \quad j=1,\dots,n
    \]
    are linearly independent in $H^1(C,T_C)$.
\end{lemma}

\begin{proof}
    With the same notations as above, using exact sequences (\ref{tangent seq}) and (\ref{cohom seq}), we observe that the cohomology classes $\zeta_j$ are obtained as the image of the boundary map $\partial_Z$ of the local functions $z^{n-1},\dots,1$ in $H^0(\mathcal{O}_Z)$ around $p$.
\end{proof}

\section{Griffiths' formula for decomposable tensors and applications}

\subsection{Griffiths formula}

Using the notation introduced above, we now recall Griffiths' formula for decomposable tensors in the case where the normal function is supported on the zeros of $\omega$ (see Eq.~(6.18) in \cite{infinitesimal}).

\begin{theorem}\label{Griff_theorem}
    Let $\pi: \mathcal{C} \rightarrow Y$ be a smooth family of complete complex curves of genus $g$, fix a base point $0 \in Y$ and let $C_0$ be the corresponding curve. Let $\nu$ be a normal function and set $\nu(0) = \mathcal{O}_{C_0}(D) \in \mathrm{Pic}^0(C_0)$ for some divisor $D$ of degree zero on $C_0$. Write $D = \sum_{i=1}^{k} a_i p_i$ with support $\mathrm{Supp}(D) = \{p_1, \dots ,p_k\}$.
    Let $\omega \in H^0(C_0,\omega_{C_0})$ be a holomorphic form, and let $Z$ denote its zero scheme, assuming that $\mathrm{Supp}(D) \subseteq Z$. Let $h \in H^0(C_0,\mathcal{O}_Z)$ such that $\partial_Z h = \xi \in H^1(C_0,T_{C_0})_\omega$. Then 
    \begin{equation}\label{Griff_formula}
    \delta\nu(0)(\xi \otimes \omega) =\sum_{i=1}^k a_ih(p_i).
    \end{equation}
\end{theorem}

We provide a sketch of the proof; for full details, see \cite[Th.~7.14]{Voisin}.

\begin{proof}[Sketch of the proof]
    Let $\tilde{\nu}$ be a lifting of $\nu$ in $\mathcal{H}^1$. 
    By~\eqref{integrale inf inv}, the infinitesimal invariant evaluated at 
    $\xi \otimes \omega \in \Ker(\overline{\nabla}^t) \subset 
    T_{Y,0} \otimes \mathcal{H}_0^{1,0}$ is
    \[
        \delta \nu (0)(\xi \otimes \omega) 
        = \int_{C} \nabla_{\xi} \tilde{\nu} \wedge \omega .
    \]
    Let $\tilde{\omega}$ be a section of $\mathcal{H}^{1,0}$ such that 
    $\tilde{\omega}(0)=\omega$. Then we have
    \begin{equation}\label{sviluppo integrale inv inf}
        \int_{C} \nabla_{\xi} \tilde{\nu} \wedge \omega
        = d_{\xi} \int_{D_y} \tilde{\omega}
          - \int_{D} \nabla_{\xi} \tilde{\omega},
    \end{equation}
    where $D_y$ denotes the divisor varying in the family.  
    Let $\Omega$ be the form associated with $\omega$ by the trivialization of the
    family. Then
    \begin{equation} \label{integrale cartan lie}
        d_{\xi} \int_{D_y} \tilde{\omega} 
        = \int_D \operatorname{int}(\tilde{\xi})(d\Omega)
          + \int_D d(\operatorname{int}(\tilde{\xi})(\Omega)),
    \end{equation}
    where $\tilde{\xi}$ is the lifting of $\xi$ in $T_Y|_{C_0}$ induced by the
    trivialization. By the Cartan--Lie formula (see \cite[Eq.~(9.8)]{Voisin_1}) the form 
    $\operatorname{int}(\tilde{\xi})(d\Omega)$ represents 
    $\nabla_\xi \tilde{\omega}$ and has cohomology class in 
    $H^1(C_0,\omega_{C_0})$ by hypothesis. Therefore, there exists a $(1,0)$-form $\beta$ and a $C^\infty$-function $h$ such that
    \[
        \operatorname{int}(\tilde{\xi})(d\Omega)|_{C_0} = \beta + dh,
    \]
    Hence
    \[
        \int_{D} \nabla_{\xi} \tilde{\omega}
        := \int_D \beta 
        = \int_D \operatorname{int}(\tilde{\xi})(d\Omega)
          - \sum_i a_i \int_D h,
    \]
    and therefore, from~\eqref{integrale cartan lie},
    \eqref{sviluppo integrale inv inf} and from the hypothesis that 
    $\operatorname{Supp}(D) \subseteq Z$, we obtain
    \[
        \delta \nu (0)(\xi \otimes \omega)
        = \sum_i a_i \int_D h
        = \sum_{i=1}^k a_i h(p_i).
    \]
\end{proof}


\begin{remark} The  condition that the support of $D$ is contained in the set of zeros of $\omega$ allows us to eliminate the  derivative that appears from the deformation of the divisor.
\end{remark}

For later purposes we want also remark the following.

\begin{remark}
    The number $\sum_{i=1}^k a_ih(p_i)$ depends only on $\xi \otimes \omega$. That is, we can express the normal function as
    \[
        \nu(y) = \mathcal{O}_{C_y}(D')
    \]
    for any degree-zero divisor $D'$ linearly equivalent to $D$, whose support is contained in the zero locus of $\omega$, without affecting the value of the infinitesimal invariant.
\end{remark}

\subsection{First results}

In this subsection, using Griffiths’ formula, we derive the following first application.

\begin{proposition}\label{prop MT}
    Let $[C] \in \mathcal{M}_g$ be a general curve of genus $g \ge 2$. Let $A, B$ be two divisors of $C$. Assume there exists an integer $n$ such that $n\cdot A$ is linearly equivalent to $n\cdot B$. Assume that the support $\mathrm{Supp}(A-B) = \{p_1, \dots, p_d\}$. Then
    \[
        h^0(K - \sum_{i=1}^d  p_i) = 0.
    \]
\end{proposition}


\begin{proof}
    Since $[C]$ is a general point of $\mathcal{M}_g$, there exists a smooth family $\pi:\mathcal{C}\to Y$ of genus-$g$ curves with $C_{y_0}=C$ for some $y_0\in Y$, such that the induced moduli map $m:Y\to\mathcal{M}_g$ is dominant. Moreover there exist $\mathcal{D} \subset \mathcal{C}$ a relative divisor such that, setting $D_y:= (\mathcal{D} \cdot C_{y})$ for every $y \in Y$ as in \eqref{ex relative divisor}, we have
    \[
        D_{y_0} = A - B = \sum_{i=1}^d a_ip_i
    \]
    and $n \cdot D_y$ is linearly equivalent to zero. Consider the normal function $\nu: Y \to \mathcal{J}(\pi)$ defined by
    \[
        \nu(y) = \mathcal{O}_{C_y}(D_y).
    \]
    Since $\mathcal{O}_{C_y}(D_y)$ is a torsion point in the Jacobian $J(C_y)$, by assumption, $\nu$ is locally constant.
    Suppose by contradiction that there exists a non zero holomorphic form $\omega \in H^0(C, \omega_{C})$ that vanishes on the points $p_1, \dots, p_n$, that is $\mathrm{Supp}(D) \subset \mathrm{Supp}(Z)$ where $Z$ is the divisor of zeros of $\omega$. Let $\theta \in H^0(\mathcal{O}_{Z})$ such that
    \[
        \theta_j(p_j) = 1, \quad \theta_j(p_i) = 0 \quad \text{for} \quad i \neq j,
    \]
    and consider the associated higher-Schiffer cohomology classes
    \[
        \zeta_{p_j} = \left[\frac{\bar{\partial}\rho_j}{\omega_j}\right]_{Dolb} \in H^1(C,T_{C})
    \]
    where $\rho_j$ is the bump function centered on the point $p_j$ for $j = 1, \dots n$. Since $y_0 \in Y$ is general and the modular map $m$ is dominant we can find tangent vectors $\xi_1, \dots, \xi_n \in T_{Y,y_0}$ such that
    \[
        \kappa_{y_0}(\xi_j) = \zeta_{p_j} \quad j = 1, \dots, n
    \]
    where $k_y$ is the Kodaira-Spencer map. By construction we have $\kappa_{y_0}(\xi_j) \cdot \omega = 0$ for every $j$. Then, the hypotheses of Theorem \ref{Griff_theorem} are satisfied and Griffiths' formula gives
    \[
        \delta \nu(y_0)(\xi_j \otimes \omega) = a_j \quad j = 1, \dots, n
    \]
    (note that in the formula (\ref{Griff_formula}) here $h = \theta_j$). If $\nu$ is locally constant, by Remark (\ref{remark loc const}), the infinitesimal invariant is zero, then we get $a_j = 0$ for $j = 1, \dots, n$. That is $D_{y_0} = 0$. This contradicts our assumption. Hence, the claim follows.
\end{proof}

\begin{remark}
    The proof of Proposition~\ref{prop MT} also shows that, under the given hypotheses, the cardinality of the support of the divisor $A-B$ is greater than $g$. This result is consistent with the Riemann-Hurwitz formula for computing the ramification index. In fact, let $f$ be the meromorphic function such that
    \[
        n\cdot A - n\cdot B = \mathrm{div}(f).
    \]
    The function $f$ defines a non-constant morphism
    \[
    f: C \longrightarrow \mathbb{P}^1,
    \]
    of some degree $l \geq 1$. The principal divisor $\mathrm{div}(f)$ has support
    \[
    \Supp(\mathrm{div}(f)) = \Supp(A - B),
    \]
    and consists of the points where $f$ has zeros or poles. Applying the Riemann-Hurwitz formula to the morphism $f: C \to \mathbb{P}^1$, we obtain:
    \[
    2g - 2 = -2l + \sum_{q \in \mathbb{P}^1} (l - \#f^{-1}(q)),
    \]
    where the last sum is the degree of the branch divisor. 
    By isolating the terms corresponding to $0$ and $\infty$, we first obtain
    \begin{align*}
    2g - 2 
        &= -2l + \bigl(2l - \#\mathrm{Supp}(A-B)\bigr) 
           + \sum_{q \in \mathbb{P}^1 \setminus \{0, \infty\}}
             \bigl(l - \#f^{-1}(q)\bigr).
    \end{align*}

    Simplifying, this gives
    \[
    \#\mathrm{Supp}(A-B) =  \sum_{q \in \mathbb{P}^1 \setminus \{0, \infty\}}
        \bigl(l - \#f^{-1}(q)\bigr) -2g + 2.
    \]
    
    By the genericity hypothesis, we have 
    $\sum_{q \in \mathbb{P}^1 \setminus \{0, \infty\}}
        \bigl(l - \#f^{-1}(q)\bigr) \geq 3g -2 $, 
    which implies the inequality
    \[
        \#\mathrm{Supp}(A - B) \ge g.
    \]
\end{remark}

As a consequence of the previous proposition, we obtain the following corollary.

\begin{corollary}\label{coro uniq K}
    Let $[C]$ be a general point of $\mathcal{M}_g$ with $g \geq 2$. Then any two distinct canonical divisors $K_1$ and $K_2$ on $C$, that are not multiples of each other, have distinct supports. In other words,
    \[
    \mathrm{Supp}(K_1) = \mathrm{Supp}(K_2) \Leftrightarrow K_1 = k\cdot K_2,
    \]
    for some integer $k$.
\end{corollary}
\begin{proof}
    Apply Proposition \ref{prop MT} with $A = K_1$, $B=K_2$ and $n=1$.
\end{proof}

In relation to the next results, we make an observation on the rank in terms of the support of the divisor associated with the normal function.

\begin{remark}
    Let $\nu$ be a normal function associated to a family
    $\pi\colon \mathcal C \to Y$ of smooth projective curves of genus $g\geq 2$,
    whose modular map is dominant onto $\mathcal M_g$.
    For a general point $y\in Y$, one may write
    \[
    \nu(y)=\mathcal O_{C_y}(D),
    \qquad
    D=\sum_{i=1}^d a_i p_i,
    \]
    where $D$ is a divisor of degree zero on $C_y$ and
    $\Supp(D)=\{p_1,\dots,p_d\}$.
    Let
    \[
    \tilde D=\sum_{i=1}^d p_i
    \]
    be the reduced divisor associated with the support of $D$.
    Let $\phi\circ\nu$ be the map defined in Definition~\ref{def rank}, and let
    \[
    \mu\colon H^0(C_y,\omega_{C_y}) \longrightarrow
    H^0(C_y,\omega_{C_y}^{\otimes 2})
    \]
    be the codifferential of $\phi\circ\nu$ at $y$.
    The behavior of $\mu$ reflects the support of the divisor $D$:
    in particular, its restriction to the subspace
    $H^0(C_y,\omega_{C_y}(-\tilde D))$ is injective.
    
    This can be seen by observing that a non--zero form
    $\omega\in H^0(C_y,\omega_{C_y}(-\tilde D))$ cannot vanish under $\mu$ in all
    directions.
    Indeed, the vanishing of $\mu(\omega)$ would imply the triviality of the
    infinitesimal invariant of $\nu$ (cf.\ Remark~\ref{remark loc const}). On the other hand, by choosing a Schiffer variation $\zeta_p$ supported at a point $p\in\Supp(D)$, the infinitesimal invariant produces a non--trivial contribution.
    As a consequence, one obtains the lower bound
    \[
    \operatorname{rk}(\nu)\geq h^0(C_y,\omega_{C_y}(-\tilde D))
    \]
    for a general point $y\in Y$.
\end{remark}

\section{Proofs of the Main Results}

In this section we prove the main theorems of the article.

\begin{theorem}\label{main th 1}
Let $\pi:\mathcal C \to Y$ be a family of complex curves of genus $g \geq 2$, and let $m:Y \longrightarrow \mathcal M_g$ be the modular map and $m(Y) \subseteq \mathcal{M}_g$ its image. Let $\psi$ be a section of the Picard bundle $\mathrm{Pic}^n(\pi)$ and 
    \[
        \nu = (2g-2)\psi - n\kappa
    \]
be the associated normal function. Then the following inequality holds,
\[
d_S(\psi) + \mathrm{rk}(\nu) + \codim_{\mathcal M_g}(m(Y)) \;\geq\; g-1.
\]
\end{theorem}

\begin{proof}
    Set $r:= \mathrm{rk}(\nu)$, $d:=d_S(\psi)$ and $c:=\codim_{\mathcal{M}_g}m(Y)$.
    Let $y \in Y$ be a general point, and set 
    \[
        \psi(y) = \OO_{C_y}(D) \in \Pic^{n}(\pi),
    \]
    for some $D$ divisor of degree $n$ on $C_y$ such that $d_S(D) = d$. Write 
    \[
        D = \sum_{i=1}^{g-1}a_ip_i, 
        \qquad 
        \Supp(D) = \{p_1, \dots, p_{g-1}\},
    \]
    and set $K$ the canonical divisor of $C_y$. Assume, for contradiction, that
    \begin{equation} \label{disug supp+rank}
        d + r +c< g - 1.
    \end{equation}
    Applying Proposition \ref{loc-cost-lungo-xi} we will prove that the rank must be greater than $r$, which contradicts the assumption.
    By (\ref{disug supp+rank}) and the Riemann–Roch theorem we have
    \[
        h^0(K - \sum_{i=1}^{d}p_i) \geq r + c + 2.
    \]
    Therefore, fixed a general point $p \in C_y \backslash \mathrm{Supp}(D)$, we can choose $r+c+1$ independent holomorphic forms $\omega_1, \dots, \omega_{r+c+1} \in H^0(C_y, \omega_{C_y}(-p_1- \cdots -p_d))$ with $\mathrm{ord}_p(\omega_j) = j$ and such that $\mathrm{Supp}(D)$ is contained in the set of zeros of $\omega_j$ for $j = 1, \dots, r+c+1$. Consider the associated higher-Schiffer cohomology classes
    \[
        \zeta_{j} = \left[\frac{\bar{\partial}\rho}{\omega_j}\right]_{Dolb} \in H^1(C_y,T_{C_y}) \quad j=1,\dots,r+c+1,
    \]
    where $\rho$ is the bump function centered at the point $p$. Consider the subspace $\Pi:= \langle \zeta_1,\dots,\zeta_{r+c+1}\rangle \subset H^1(C_y,T_{C_y})$  generated by $\zeta_1, \dots ,\zeta_{r+c+1}$. By Lemma \ref{schiffer indip} these vectors are linearly independent, hence
    \[
    \dim \Pi = r+c+1.
    \]
    By Grassmann formula we have
    \[
    \dim (\Pi \cap \mathrm{d}m_y(T_{Y,y})) \geq r+1,
    \]
    where $\mathrm{d}m_y$ is the differential of the modular map at $y \in Y$.
    Then, there are $r+1$ linearly independent higher Schiffer classes $\zeta_{1}, \dots, \zeta_{r+1}$ in $\Pi \cap \mathrm{d}m_y(T_{Y,y})$, 
    where (without loss of generality) we assume that they correspond to the first $r+1$ indices. 
    Since $y \in Y$ is a general point we can find vectors $\xi_j \in T_{Y,y}$ such that $\kappa_y(\xi_j) = \zeta_j$ for $j= 1, \dots, r+1$, where $\kappa_y$ is the Kodaira-Spencer map.
    Set $\tilde{\Pi}:= \langle \xi_1,\dots,\xi_{r+1}\rangle \subset T_{Y,y}$ the subspace generated by these vectors.

    Suppose there exists a vector $\xi = \sum_{i = 1}^{r+1} \alpha_i \xi_i \in \tilde{\Pi}$ such that $\nu$ is locally constant along $\xi$.
    By construction $\kappa_y(\xi) \cdot \omega_{r+1} = \sum_{i = 1}^{r+1}\alpha_i \zeta_i \cdot \omega_{r+1} = 0$. Set $Z_{r+1} := \mathrm{div}(\omega_{r+1})$ as representative of the canonical divisor and write
    \[
        \nu(y) = \mathcal{O}_{C_y}((2g-2)D - n\cdot Z_{r+1}).
    \]
    By Remark \ref{remark loc const}, applying the Griffiths' formula we have
    \[
        0 = \delta \nu(y) (\xi \otimes \omega_{r+1}) = -\alpha_{r+1}\cdot n(r+1),
    \]
    which implies $\alpha_{r+1} = 0 $. Therefore, $\xi = \sum_{i=1}^{r} \alpha_i \xi_i$. We now proceed by reverse induction on $j$. At each step, we replace the representative of the canonical divisor by setting $Z_j := \mathrm{div}(\omega_j)$, and evaluate the infinitesimal invariant on $\xi \otimes \omega_j$. Observe that
    \[
        \sum_{i = 1}^{j} \alpha_i \zeta_i \cdot \omega_j = 0,
    \]
    so the same argument as above shows that $\alpha_j = 0$. Repeating this process inductively for $j = r, r - 1, \dots, 1$, we conclude that $\alpha_i = 0$ for all $i = 1, \dots, r+1$, and hence $\xi = 0$. This proves that for every $\xi \in \tilde{\Pi}$ and for every local lifting $\tilde{\nu} \in \mathcal{H}$ 
    \[
    \nabla_\xi \tilde{\nu} \neq 0,
    \]
    that is $\nu$ is not locally constant along $\xi$. By Proposition \ref{loc-cost-lungo-xi} this shows that $\nu$ has rank greater than $r$, hence a contradiction.
\end{proof}

\begin{remark}
    The result of Theorem~\ref{main th 1} is implied by \cite[Th.~2.3.1]{Betti_map}. However, our argument applies to analytic families of curves, namely to analytic open subsets of $\mathcal{M}_g$.
\end{remark}

Observe that if the modular map $m$ is dominant, the inequality becomes
\[
    d_S(\psi) + \mathrm{rk}(\nu) \;\geq\; g-1.
\]
In this case we have a characterization for $\psi$ when $\nu$ is locally constant and $d_S(\psi) = g-1$. More precisely, we have the following.

\begin{theorem}\label{theo-spin-section}
    Let $\pi:\mathcal{C} \to Y$ be a family of complex curves of genus $g \geq 2$, and assume that the image of the modular map $m: Y \to \mathcal{M}_g$ contains an analytic open subset.  
    Let $\psi$ be a section of the Picard bundle $\mathrm{Pic}^n(\pi)$, and consider the associated normal function
    \[
        \nu = (2g-2)\psi - n\kappa .
    \]
    If $\mathrm{rk}(\nu) = 0$, i.e.\ $\nu$ is locally constant, then $d_S(\psi) = g-1$ if and only if, for a general point $y \in Y$ and every divisor $D$ such that $\psi(y) = \mathcal{O}_{C_y}(D)$ with $d_S(D) = d_S(\psi)$, there exists a holomorphic form $\omega \in H^0(C_y,\omega_{C_y})$ with divisor $K := \sum_{i=1}^{g-1} b_i p_i$ and support $\mathrm{Supp}(K) = \{p_1,\dots, p_{g-1}\}$, such that $\mathrm{Supp}(D) = \mathrm{Supp}(K)$, such that one of the following holds:
    \begin{enumerate}
        \item $D$ is a multiple of the divisor $\sum_{i=1}^{g-1} p_i$, i.e.\ $\psi$ is a multiple of an odd spin section;
        \item $D$ is a multiple of $K$, i.e.\ $\psi$ is a multiple of the canonical section.
    \end{enumerate}
    In either case, $\mathrm{Supp}(D) = \mathrm{Supp}(K)$.
    
\end{theorem}

\begin{proof}
    We use the same notation and conventions of the proof of the Theorem \ref{main th 1}.
    Let $y \in Y$ be a general point, and set 
    \[
        \psi(y) = \OO_{C_y}(D) \in \Pic^{n}(\pi),
    \]
    for some $D$ divisor of degree $n$ on $C_y$ such that $d_S(D) = d_S(\psi) = g-1$. Write 
    \[
        D = \sum_{i=1}^{g-1}a_ip_i, 
        \qquad 
        \Supp(D) = \{p_1, \dots, p_{g-1}\},
    \]
    and set $K$ the canonical divisor of $C_y$.
    By Riemann-Roch we have
    \[
        h^0(K - \sum_{i=1}^{g-1}p_i) = 1.
    \]
    Then, there exists a holomorphic form $\omega \in H^0(C_y,\omega_{C_y})$ such that $\mathrm{Supp}(D)$ is contained in the set of zeros of $\omega$. Let $K = \mathrm{div}(\omega) = \sum_{i = 1}^{g-1} b_i p_i + \sum_{j = 1}^{l} c_j q_j$ and use this as a representative of the normal function $\nu$,
    \[
        \nu(y) = \mathcal{O}_{C_y}((2g-2)D - n\cdot K).
    \]
    Let us show that $l = 0$. Let $\zeta_{q_j}$ be the higher-Schiffer cohomology class centered on $q_j$, and let $\xi_j \in T_{Y,y}$ the corresponding vector such that $\kappa_y(\xi_j) = \zeta_{q_j}$. Applying the Griffiths' formula on $\xi_j \otimes \omega$ we have
    \[
        \delta \nu(y) (\xi_j \otimes \omega) = c_j.
    \]
    Since $\nu$ is locally constant, we must have $c_j = 0$ for $j = 1, \dots, l$, so that 
    \[
        K = \sum_{i = 1}^{g-1} b_i p_i.
    \]
    Moreover, the divisor $K$ is unique by Corollary \ref{coro uniq K}. 
    We distinguish two cases. If $\gcd(b_1, \dotsc, b_{g-1}) = 2$, then $b_i = 2$ for every $i$,
    \[
        K = \sum_{i = 1}^{g-1} 2 p_i = 2 \left(\sum_{i = 1}^{g-1} p_i\right).
    \]
    That is $\omega_{C_y} = \mathcal{L}^2$ where $\mathcal{L} = \mathcal{O}_{C_y}(\sum_{i = 1}^{g-1}  p_i)$. Then the normal function has the form
    \[
        \nu(y) = \mathcal{O}_{C_y}\left(\sum_{i = 1}^{g-1}\left[(2g-2)a_i  -2n\right]\cdot p_i\right).
    \]
    Applying the Griffiths' formula on the higher-Schiffer cohomology classes $\zeta_{p_i}$ we get the relations
    \[
        (2g-2)a_i =2 n \quad i=1,\dots,g-1.
    \]
    That is, there exists an integer $k$ such that $a_i = k$ for every $i$. This yields
    \[
        D = \sum_{i = 1}^{g-1} k \cdot p_i = k \cdot\left( \sum_{i = 1}^{g-1}  p_i\right).
    \]
    On the other hand if $\gcd(b_1, \dotsc, b_{g-1}) = 1$ then there exists and index $i$ such that $b_i = 1$. Let us suppose $i = 1$. With the same argument as before, applying the Griffiths' formula on the higher-Schiffer $\zeta_{p_1}$ we get
    \[
        (2g-2) a_1 = n.
    \]
    Set $k:= a_1$. Hence, proceeding in the same way, we obtain
    \[
        D = \sum_{i = 1}^{g-1} a_i  p_i = k \left(\sum_{i = 1}^{g-1}b_i p_i\right) = k\cdot K.
    \]
    This concludes the proof.
    \end{proof}  

As a consequence of Theorem \ref{theo-spin-section} we obtain the following corollary.

\begin{corollary}\label{coro-spin}
    Under the assumptions of Theorem~\ref{theo-spin-section}, if $\psi \in \mathrm{Pic}^{g-1}(\pi)$ is a section such that, for a general point $y \in Y$, $h^0(C_y,\psi(y)) > 0$ and $\nu$ is locally constant, then $\psi$ coincides with the odd spin section.
\end{corollary}

An application of the Theorem \ref{main th 1} is the following. Let $\mathcal{K}_g \subset \mathcal{M}_g$ denote the locus of $k$-gonal curves, namely
\[
    \mathcal{K}_g \;=\; \Big\{\, [C] \in \mathcal{M}_g \;\Big|\; \exists f : C \to \mathbb{P}^1 \text{ of degree } k \,\Big\}.
\]
It is well known that $\mathcal{K}_g$ is an irreducible subvariety of dimension $\dim\, \mathcal{K}_g = 2g + 2k - 5$ (see \cite{Fulton}).
Assuming $g > 2k - 2$,its codimension in $\mathcal{M}_g$ is
\[
    \operatorname{codim}_{\mathcal{M}_g}(\mathcal{K}_g) = g - 2k + 2.
\]

Let $\pi : \mathcal{C} \to Y$ be a family of $k$-gonal curves whose modular map factors generically through $\mathcal{K}_g$, and let 
\[
    \psi \in \operatorname{Pic}^k(\pi)
\]
be the section corresponding to the $g^1_k$ defining the $k$-fold map to $\mathbb{P}^1$. Since for the generic curve the ramification is simple we have,
\[
    d_S(\psi) = k -1.
\]
Then, applying Theorem~\ref{main th 1} to the associated normal function $\nu$, we obtain 
\begin{equation*}
    (k-1) + (g-2k+2) + \operatorname{rk}(\nu) \geq g-1,
\end{equation*}
giving the lower bound
\begin{equation}\label{eq:rank-k-gonal}
    \operatorname{rk}(\nu) \;\ge\; k - 2.
\end{equation}

Note that in the hyperelliptic case $k = 2$, the inequality \eqref{eq:rank-k-gonal} becomes an equality.  
Indeed, for hyperelliptic curves the unique $g^1_2$ is induced by the canonical linear system, and the rank of $\nu$ is exactly $0$.

\section{Plane Curves}
    In this section we turn to plane curves and apply the methods established in the previous sections in order to prove the Proposition \ref{prop_d_4_5 introduz}.
    Let $\mathcal{V}_d \subset\mathcal{M}_g$ be the subvariety of the smooth plane curves of degree $d$ and genus $g = \frac{(d-1)(d-2)}{2}$.
    Let $\pi: \mathcal{C} \to Y$ be a family of plane curves of degree $d$. For us, this means that there exists a morphism
    \[
        \Phi: \mathcal{C} \longrightarrow \mathbb{P}^2 \times Y
    \]
    such that, for every $y \in Y$, the induced map
    \[
        \Phi_y: C_y \longrightarrow \mathbb{P}^2 \times \{y\}
    \]
    is an embedding, and $\Phi_y(C_y) \subset \mathbb{P}^2$ is a smooth plane curve of degree $d$. Let
    \[
        m: Y \to \mathcal{V}_d \subset\mathcal{M}_g
    \]
    be the modular map.
    For simplicity, we will identify $C_y$ with its image $\Phi_y(C_y)$.

\subsection{The Case of Quartic Plane Curves}

For quartic plane curves, Theorem~\ref{theo-spin-section} applies directly, yielding an immediate geometric interpretation.

\begin{proposition}\label{prop quartic}
    Let $C$ be a very general plane curve of degree $4$, and let $D \subset \mathbb{P}^2$ be a curve. Suppose that $C$ and $D$ intersect in exactly two points $p_1$ and $p_2$. Then $p_1$ and $p_2$ lie either on a bitangent line or on a flex line.
\end{proposition}

\begin{proof}
    Let $\pi: \mathcal{C} \to Y$ be a family of plane quartic curves. Assume that the modular map
    \[
        m: Y \to \mathcal{M}_3
    \]
    is dominant.
    For a general point $y_0 \in Y$ we write $C = C_{y_0}$. If $D \subset \mathbb{P}^2$ is a curve of degree $n$ intersecting $C$, then the corresponding line bundle $\mathcal{O}_C(D)$ defines a section of $\psi$ of $\mathrm{Pic}^{4n}(\pi)$ given by
    \[
        \psi = \Phi^*\mathcal{O}_{\mathbb{P}^2}(n) \;\cong\; n \cdot \Phi^*\mathcal{O}_{\mathbb{P}^2}(1) \;=\; n \cdot \kappa,
    \]
    and can be represented at $y_0 \in Y$ by
    \[
        \psi(y_0) = \mathcal{O}_{C}(a_1p_1 + a_2p_2),
    \]
    where $a_1 + a_2 = 4n$.  
    Then the associated normal function
    \[
        \nu = 4\psi - 4n\kappa = 0,
    \]
    hence locally constant. By Theorem~\ref{theo-spin-section}, there exists a holomorphic form $\omega \in H^0(C,\omega_C)$ representing the canonical divisor $K$ of $C$ with support 
    $\{p_1, p_2\}$, such that, up to exchanging $p_1$ and $p_2$, we must have
    \[
        a_1p_1 + a_2p_2 = nK =  n(2p_1 + 2p_2) 
    \]
    or
    \[
        a_1p_1 + a_2p_2 = nK =  n(3p_1 + p_2) 
    \]
    This proves that the two points $p_1$ and $p_2$ lie either on a bitangent line or on a flex line of $C$, as required.
\end{proof}

\subsection{The Case of Quintic Plane Curves}

We begin by recalling several classical facts concerning the cohomology of smooth plane curves.  
Let $(z_0,z_1,z_2)$ be homogeneous coordinates on $\PP^2$, and let
\[
    S = \bigoplus_{d \ge 0} S^d = \C[z_0,z_1,z_2]
\]
denote its homogeneous coordinate ring.  
Let $C \subset \PP^2$ be a smooth complex plane curve of degree $d \ge 4$, defined by a homogeneous polynomial $f$ of degree $d$
\[
    C = V(f) = \{ [z_0:z_1:z_2] \in \PP^2 \mid f(z_0,z_1,z_2)=0 \}.
\]

The \emph{Jacobian ideal} of $f$ is defined as the ideal generated by the partial derivatives of $f$,
\[
    J(f) = \big( \partial_{z_0} f,\ \partial_{z_1} f,\ \partial_{z_2} f \big) \subset S,
\]
and we denote by
\[
    R = S / J(f), \qquad R^k = (S/J(f))_k,
\]
the associated Jacobian ring and its degree–$k$ graded component.  
It is well known that $R$ is a standard Artinian Gorenstein algebra (SAGA) with socle in degree $3d - 6$; in particular, it satisfies the Gorenstein duality
\begin{equation}\label{gorenstein-duality}
    (R^d)^* \cong R^{2d - 6}.
\end{equation}

In particular,the vector space \(R^d\) can be identified with the Zariski tangent space to the orbifold moduli space of smooth plane curves of degree \(d\) (see \cite[Chap.~2, Sec.~6]{Voisin} for further details).
 
By the adjunction formula, $\omega_C \cong \OO_C(d-3)$, and therefore
\[
    H^0(C, \omega_C^{\otimes 2}) \cong H^0\big(C, \OO_C(2d-6)\big).
\]
The inclusion $R^{2d-6} \hookrightarrow H^0(C, \omega_C^{\otimes 2})$ together with Serre duality,
\begin{equation}\label{serre duality plane curves}
    \langle\, ,\, \rangle \colon 
        H^1(C, T_C) \times H^0(C, \omega_C^{\otimes 2}) 
        \longrightarrow \C,
\end{equation}

realizes the Gorenstein duality \eqref{gorenstein-duality}.  
Consequently,
\begin{equation}\label{perp jacobian ideal}
    R^d
    = \{\, \xi \in H^1(C, T_C) \mid 
         \langle \xi , w \rangle = 0
         \text{ for all } w \in J_{2d-6} \,\}
    = J_{2d-6}^{\perp},    
\end{equation}

where $J_{2d-6}$ denotes the degree–$(2d-6)$ component of the Jacobian ideal.

We now recall a Lemma that will be of use in the sequel. This statement is a special instance of the weak Lefschetz property for the Gorenstein ring $R^d$ with $d = 5$. For a comprehensive treatment of the general case, we refer the reader to \cite{Lef_prop} or \cite{Stanley_graded}.

\begin{lemma}\label{lem:injective-L}
    Let $C$ a generic smooth plane curve of degree $5$, give by the zeros of a homogeneous polynomial
    $F \in S^5$.    
    Let $L = \alpha z_0 + \beta z_1 + \gamma z_2 \in S_1$ be a linear form with $\alpha\beta\gamma \neq 0$.  
    Then the multiplication map
    \[
        \cdot L : R^4 \longrightarrow R^5 \quad [f] \mapsto [L \cdot f],
    \]
    is an isomorphism.
\end{lemma}

\begin{proof}
    Since the dimensions of the vector spaces $R^4$ and $R^5$ are independent of the polynomial $F$, the thesis is a open property of $F$, hence it is sufficient to check it for a particular $F$. Then, take $F$ the Fermat polynomial $F = z_0^5 + z_1^5 + z_2^5$. In this case, the proof reduces to a direct linear–algebra computation, obtained by writing explicit bases for $R^4$ and $R^5$.
\end{proof}

Let $C \subset \mathbb{P}^2$ a generic smooth plane curve of degree $5$. Then $\omega_C \cong \mathcal{O}_C(2)$ and $H^0(C,\omega_C) \cong R^2 = S^2$. Let $h, t \in H^0(C, \mathcal{O}_C(1))$ be two different non zero sections.
Denote by $H$ and $T$ the zero divisors in $C$ of $h$ and $t$, respectively. Under the natural isomorphism
    \[
        \varphi:H^0(C, \mathcal{O}_C(1)) \cong H^0(\mathbb{P}^2, \mathcal{O}_{\mathbb{P}^2}(1)),
    \]
the sections $h$ and $t$ correspond to linear forms defining distinct lines
$\ell_h, \ell_s \subset \mathbb{P}^2$,
\[
    \ell_h = V(\varphi(h)), \quad \ell_t = V(\varphi(t)),
\]
such that
\[
    H = (\ell_h \cdot C), \quad T = (\ell_t \cdot C).
\]
Let $\omega \in H^0(C,\omega_C) \cong H^0(C,\mathcal{O}_C(2))$ be the holomorphic form corresponding to
    \[
        \omega \equiv h \cdot t,
    \]
with support $\mathrm{Supp}(\operatorname{div}(\omega))
        = \mathrm{Supp}(H) \cup \mathrm{Supp}(T)$.

\begin{proposition}\label{prop existence schiffer}
    Let $C$ be a generic smooth plane curve of degree $5$. Let $t,h \in H^0(C,\mathcal{O}_C(1))$ be general sections such that,
    \begin{enumerate}[(a)]
        \item the multiplication map by $t$, $R^4 \stackrel{\cdot t}\rightarrow R^5$, is an isomorphism,
        \item $\mathrm{Supp}(T) \cap \mathrm{Supp}(H) = \emptyset$.
    \end{enumerate}
    Then, for every point $p \in \mathrm{Supp}(T)$, there exist deformations $\zeta, \eta \in H^1(C,T_C)$ such that, 
    \begin{enumerate}
        \item $\zeta \cdot\omega = 0$,
        \item $\zeta  = \zeta_{p} + \eta \in R^5$,
    \end{enumerate}
    where $\zeta_{p}$ is the higher Schiffer cohomology class centered on a point $p \in \mathrm{Supp}(T)$.
\end{proposition}

\begin{proof}
    Let $p \in \mathrm{Supp}(T)$ be a point. By the genericity of $t$ write
    \[
        \mathrm{Supp}(T) = \ell_t \cap C
        = \{p, q_1, q_2, q_3, q_4\}.
    \]
    Consider the subspaces of $H^1(C,T_C)$
    \[
        \Pi := 
            \langle \zeta_{p}, \zeta_{q_1}, \zeta_{q_2}, 
                    \zeta_{q_3}, \zeta_{q_4} \rangle,
        \qquad
        \tilde{\Pi} := 
            \langle \zeta_{q_1}, \zeta_{q_2}, 
                    \zeta_{q_3}, \zeta_{q_4} \rangle,
    \]
    generated by the corresponding higher Schiffer variations.  
    By Lemma~\ref{schiffer indip}, these classes are linearly independent, and therefore $\dim \Pi = 5, \dim \tilde{\Pi} = 4$. Since $\dim R^5 = 12$, the Grassmann formula implies
    \[
        \dim(\Pi \cap R^5)
            = \dim \Pi + \dim R^5 
              - \dim(\Pi + R^5)
            \ge 2.
    \]
    Our goal is to show that
    \[
        \dim(\tilde{\Pi} \cap R^5) = 1.
    \]
    If this holds, then there exists a linear combination
    \[
        \zeta_{p} + \sum_{i=1}^4 \alpha_i \, \zeta_{q_i} \in \Pi \cap R^5,
    \]
    which, by construction, satisfies $\zeta \cdot \omega = 0$.
    Recall by Serre duality \eqref{serre duality plane curves}, one has
    \[
        \tilde{\Pi}^{\perp}
          = H^0\left(C, \omega_C^{\otimes 2}(-q_1 - \cdots - q_4)\right).
    \]
    Because $q_1,\dots,q_4$ are general points on $C$, it follows that $\dim(\tilde{\Pi}^{\perp}) = 11$.
    Moreover, by \eqref{perp jacobian ideal}, one has $(R^5)^{\perp} = J_4$, and then $\dim((R^5)^{\perp}) = 3$.
    Assume, for contradiction, that $\dim(\tilde{\Pi} \cap R^5) \ge 2$.  
    Passing to orthogonal complements and applying Grassmann’s formula, we have
    \[
        \dim\left((\tilde{\Pi} \cap R^5)^{\perp}\right)
             = \dim(\tilde{\Pi}^{\perp} + (R^5)^{\perp})
             = \dim(\tilde{\Pi}^{\perp}) + \dim((R^5)^{\perp})
               - \dim(\tilde{\Pi}^{\perp} \cap (R^5)^{\perp})
             \le 13.
    \]
    Hence
    \[
        \dim(\tilde{\Pi}^{\perp} \cap (R^5)^{\perp}) \ge 1.
    \]
    This means that there exists a nonzero partial derivative
    \[
        f_i = \partial_{z_i} f \in J_4
    \]
    vanishing at the four points $q_1,\dots,q_4$.  
    Without loss of generality, we may take $i = 0$.
    We distinguish two cases.

    \textbf{Case 1:} $f_0(p) = 0$.  
    Then $f_0$ vanishes at five points in linear general position.  
    Since $\deg(f_0)=4$, it follows that $f_0$ must factor as
    \[
        f_0 = Q_3 \cdot L,
    \]
    with $\deg Q_3 =3$ and $\deg L =1$.  
    Since $C$ is general, its partial derivatives do not admit nontrivial factorizations, yielding a contradiction.

    \textbf{Case 2:} $f_0(p) \neq 0$.  
    Both $f_0$ and $t$ vanish at the points $q_1,\dots,q_4$, hence there exist polynomials $Q_4$ of degree $4$ and $L$ of degree $1$ such that
    \[
        f = Q_4 \cdot t + L \cdot f_0.
    \]
    This implies that $Q_4 \cdot t \in J_5$.  
    Since $t$ is general, Lemma~\ref{lem:injective-L} forces $Q_4$ either to vanish or to lie in the span of the partial derivatives of $f$.  
    In either case this contradict the generality of $C$. Thus $\dim(\tilde{\Pi} \cap R^5) = 1$, completing the proof.
\end{proof}

Before stating the next result, we apply Theorem~\ref{main th 1} to the subvariety $\mathcal{V}_5 \subset \mathcal{M}_6$,
parametrizing smooth plane curves of degree $5$.  
Let $\psi$ and $\nu$ be as in the hypotheses of Theorem~\ref{main th 1}.  
Since $g = 6$ and
\[
    \operatorname{codim}_{\mathcal{M}_6}(\mathcal{V}_5) = 3,
\]
if the function $\nu$ is locally constant along $\mathcal{V}_5$, then Theorem~\ref{main th 1} yields
\begin{equation}\label{bound degree 5}
    d_S(\psi) \ge 2.
\end{equation}

We can now refine the bound \eqref{bound degree 5} and establish a generalization of Proposition~\ref{prop quartic} and prove the Proposition \ref{prop_d_4_5 introduz} in the case of degree $5$.

\begin{theorem}
    Let $\pi: \mathcal{C} \to Y$ be a family of plane curves of degree $5$. Assume that the modular map $m: Y \to \mathcal{M}_6$ contains an analytic open subset of $\mathcal{V}_5$.  Let $\psi$ be a section of the Picard bundle $\mathrm{Pic}^n(\pi)$ and
    \[
        \nu = 10\psi - n\kappa
    \]
    be the associated normal function. If $\operatorname{rk}(\nu) = 0$, i.e. $\nu$ is locally constant, then 
    \[
        d_S(\psi) \geq 3.
    \]
    Moreover, if the equality holds, then the points of $\mathrm{Supp}(\psi)$ must be collinear, and they lie either on a bitangent line or on a flex line.
\end{theorem}

\begin{proof}
    Let $y \in Y$ be a general point, and write 
    \[
        \psi(y) = \OO_{C_y}(D) \in \Pic^{n}(\pi),
    \]
    where $D = \sum_{i=1}^d a_i p_i$ is a divisor of degree $n$ on $C_y$ with support $\Supp(D) = \{p_1, \dots, p_d\}$.  
    By~\eqref{bound degree 5} assume, for contradiction, that $d = 2$. Then
    \[
        D = a_1 p_1 + a_2 p_2,
        \qquad
        \Supp(D) = \{p_1, p_2\}.
    \]
    
    Choose general sections $h, t \in H^0(C_y, \OO_{C_y}(1))$ satisfying the hypotheses of Proposition~\ref{prop existence schiffer} and such that
    \[
        p_1 \in \Supp(H), 
        \qquad
        p_2 \in \Supp(T),
    \]
    where $H = \operatorname{div}(h)$ and $T = \operatorname{div}(t)$ are the corresponding divisors on $C_y$.
    Using the isomorphism $H^0(C_y,\omega_{C_y}) \cong H^0(C_y,\mathcal{O}_{C_y}(2))$, we identify the canonical divisor with $K = 2H$. Hence the normal function can be
    written as
    \[
        \nu(y) = \mathcal{O}_{C_y}(10D - nK)
               = \mathcal{O}_{C_y}(10D - 2nH).
    \]

    Let $\omega \equiv h\cdot t \in H^0(C_y,\omega_{C_y})$ be the holomorphic form associated to the divisor $H + T$.
    By Proposition~\ref{prop existence schiffer}, there exists a class
    \[
        \zeta = \zeta_{p_2} + \eta \in R^5,
    \]
    where $\zeta_{p_2}$ is the higher Schiffer variation centered at $p_2$. 
    Let $\xi \in T_{Y,y}$ be such that $\kappa_y(\xi) = \zeta$.  
    Since $\nu$ is locally constant, the Griffiths infinitesimal invariant yields
    \[
        0 = \delta\nu(y)(\xi \otimes \omega) = a_2.
    \]
    Thus $a_2 = 0$, and therefore $D = a_1 p_1$,
    so $d_S(D) = 1$, contradicting~\eqref{bound degree 5}.
    
    For the second part of the theorem, let $d = 3$. Then 
    \[
        D = a_1 p_1 + a_2 p_2 + a_3 p_3, 
        \qquad 
        \Supp(D) = \{p_1, p_2, p_3\}.
    \]
    Assume, for contradiction, that $p_1, p_2, p_3$ are not collinear.  
    Let $\ell \subset \PP^2$ be the line through $p_1$ and $p_2$, and let 
    $r \subset \PP^2$ be a general line through $p_3$ such that it does not meet $\ell$ on $C_y$.  
    Set
    \[
        L = (\ell \cdot C_y), 
        \qquad 
        R = (r \cdot C_y),
    \]
    and take $K = 2L$ as representative of the canonical divisor in the normal function.
    
    Consider $\omega \in H^0(C_y,\omega_{C_y})$ be the holomorphic form associated to the divisor $L+R$. Applying the same procedure as before, by taking the Schiffer class 
    \[
        \zeta = \zeta_{p_3} + \eta \in R^5,
    \]
    we obtain that $a_3 = 0$, contradicting the first part of the theorem. Then, $d \geq 3$. If $d = 3$, we proved that the points $p_1,p_2$ and $p_3$ must lie on the same line. That is
    \[
        D = nD',
    \]
    with $D'$ divisor on $C_y$ cut by a line, with $\deg(D') = 5$ and $\mathrm{Supp}(D') =\{p_1,p_2,p_3\}$. Hence, up to change the order of the points, we must have
    \[
        D' = 2p_1 + 2p_2 + p_3 \quad \text{or} \quad
        D' = 3p_1 + p_2 + p_3.
    \]
    That is, $p_1,p_2$ and $p_3$ must lie either on a bitangent line or on a flex line.

\end{proof}

\section*{Funding}
L.Fassina and G.P. Pirola are members of GNSAGA (INdAM).

\newpage
\printbibliography

\end{document}